\documentclass[11pt,twoside,a4paper]{article}
\usepackage{amsmath,amssymb,amsthm}
\usepackage{graphicx}
\usepackage[english]{babel}
\usepackage{scrlayer-scrpage}
\usepackage[explicit]{titlesec}
\usepackage{libertine}
\usepackage{tikz}
\usetikzlibrary{calc}
\usepackage{tikzpagenodes}
\usepackage{hyperref}
\usepackage[margin=1.5in]{geometry}

\theoremstyle{plain}
\newtheorem{St}{Theorem}[section]
\newtheorem{Le}[St]{Lemma}
\newtheorem{Gev}[St]{Corollary}
\theoremstyle{definition}
\newtheorem{Ex}[St]{Example}
\newtheorem{Def}[St]{Definition}
\newtheorem{Opm}[St]{Remark}

\DeclareMathOperator\pg{\mathrm{PG}}
\DeclareMathOperator\ag{\mathrm{AG}}

\newcommand{\cL}{\mathcal{L}}
\newcommand{\gauss}[2]{\genfrac{[}{]}{0pt}{}{#1}{#2}}

\begin{document}
	\title{A modular equality for Cameron-Liebler line classes in projective and affine spaces of odd dimension}

\footnotetext{$\dagger$ Department of Mathematics and Data Science, 
Vrije Universiteit Brussel (VUB),  Pleinlaan 2, B--1050 Brussels, 
Belgium  (Email: Jan.De.Beule@vub.be, Jonathan.Mannaert@vub.be; 
\url{http://homepages.vub.ac.be/~jdbeule}, \url{http://homepages.vub.ac.be/~jonmanna}).}%
\author{J. De Beule $^\dagger$ and J. Mannaert$^\dagger$}

\maketitle
\begin{abstract}
In this article we study Cameron-Liebler line classes in $\pg(n,q)$ and $\ag(n,q)$, objects also known as boolean degree
one functions. A Cameron-Liebler line class $\cL$ is known to have a parameter $x$ that depends on the size of  $\cL$. 
One of the main questions on Cameron-Liebler line classes is the (non)-existence of these sets for certain parameters $x$. 
In particularly it is proven in \cite{MetschAndGavrilyuk} for $n=3$, that the parameter $x$ should satisfy a modular equality. 
This equality excludes about half of the possible parameters. We generalize this result to a modular equality 
for Cameron-Liebler line classes in $\pg(n,q)$, and $\ag(n,q)$ respectively. Since it is known that a Cameron-Liebler line class 
in $\ag(n,q)$ is also a Cameron-Liebler line class in its projective closure, we end this paper with proving that the modular 
equality in $\ag(n,q)$ is a stronger condition than the condition for the projective case.
\end{abstract}

\section{Introduction}\label{sec:intro}

The study of Cameron-Liebler sets of $k$-spaces, (or Cameron-Liebler $k$-sets for short), 
in finite affine or projective spaces of general dimension, originated from the study 
Cameron-Liebler line classes in $\pg(3,q)$. Cameron and Liebler in 
\cite{Cameron-Liebler} studied irreducible collineation groups of $\pg(d,q)$, having
equally many point orbits as line orbits. This work resulted in several equivalent definitions of
line sets (that were later called Cameron-Liebler line classes) having particular combinatorial properties, or, 
equivalently, having very distinctive algebraic properties. One of the questions posed in \cite{Cameron-Liebler}
is whether {\em non-trivial} examples of Cameron-Liebler line classes exist. The conjecture answering this
question negatively, was later shown to be not true by a constructive result. However, non-trivial Cameron-Liebler
line classes are very rare. Hence not surprisingly, there are a lot of non-existence results for particular values
of the {\em parameter} of the Cameron-Liebler line class, the number which determines its main 
combinatorial characteristic. More precisely, a Cameron-Liebler line class with parameter $x$ is
a set of lines of $\pg(3,q)$ meeting every spread in exactly $x$ lines. Trivial examples with parameter $1$ 
are the set of lines through a fixed point and dually the set of lines in a fixed plane. 
One of the strongest non-existence results on Cameron-Liebler line classes is the following result. 

\begin{St}\cite[Theorem 1.1]{MetschAndGavrilyuk}\label{MetschRes}
Suppose that \(\mathcal{L}\) is a Cameron-Liebler line class with parameter \(x\) of $\pg(3,q)$.  
Then for every plane and every point of $\pg(3,q)$, 
\begin{equation}\label{Metsch}
\binom{x}{2} +m(m-x)\equiv 0 \mod (q+1),
\end{equation}
where \(m\) is the number of lines of \(\mathcal{L}\) in the plane, respectively through the point.
\end{St} 

Cameron-Liebler line classes in $\pg(3,q)$ have been generalized to Cameron-Liebler line classes in $\pg(n,q)$, 
see \cite{DrudgeThesis}, and very recently  to Cameron-Liebler $k$-sets in $\pg(n,q)$, $n \geq 2k+1$, 
see \cite{Jozefien}. In both cases, classification of such objects based on their parameter remains of 
great interest. Cameron-Liebler line classes and Cameron-Liebler 
$k$-sets have been defined and studied in affine spaces as well, see \cite{Me1,Me2}. Once again, classification 
of these objects has been motivating to study restrictions on the parameter. From \cite{Me1} we recall the following
result, which is in principle an easy corollary of Theorem~\ref{MetschRes}.

\begin{St}[\cite{Me1}, Corollary 4.3]\label{GevKlausMain}
Suppose that $\mathcal{L}$ is a Cameron-Liebler line class in $\ag(3,q)$ with parameter $x$. Then
\begin{equation}\label{CL1}
x(x-1) \equiv 0 \mod 2(q+1) \,.
\end{equation}
\end{St}

Such modular conditions on the parameter imply non-existence for particular values of $x$. This observation
has been the motivation to investigate similar results for Cameron-Liebler line classes in $\pg(n,q)$ and $\ag(n,q)$. 
The main results of this paper are the following theorems. 

\begin{St}
Suppose that $\mathcal{L}$ is a Cameron-Liebler line class with parameter $x$ in $\pg(n,q)$, $n\geq 7$ odd. Then for
any point $p$,
\[
x(x-1)+2\overline{m}(\overline{m}-x)\equiv 0 \mod (q+1)\,,
\]
where $\overline{m}$ is the number of lines of $\mathcal{L}$ through $p$.
\end{St}

\begin{St}
Suppose that $\mathcal{L}$ is a Cameron-Liebler line class in $\ag(n,q)$, $n \geq 3$, with parameter $x$, then
$$x(x-1)\frac{q^{n-2}-1}{q-1}\equiv0 \mod 2(q+1).$$
\end{St}

\section{Preliminaries}

Recall that a line spread of $\pg(3,q)$, respectively $\ag(3,q)$ is a set of lines of $\pg(3,q)$, respectively $\ag(3,q)$
partitioning the point set of $\pg(3,q)$, respectively $\ag(3,q)$. We start with formally defining Cameron-Liebler 
line classes are their generalizations. 

\begin{Def}
A Cameron-Liebler line class with parameter $x$ in $\pg(3,q)$, respectively $\ag(3,q)$, is a set $\mathcal{L}$ 
of lines of $\pg(3,q)$, respectively $\ag(3,q)$, such that $|\mathcal{L} \cap \mathcal{S}| = x$ for every spread $\mathcal{S}$
of $\pg(3,q)$, respectively $\ag(3,q)$.
\end{Def}

Classical examples of Cameron-Liebler line classes in $\pg(3,q)$ are (1) the empty set, with parameter $x=0$, 
(2) all lines through a fixed point, with parameter $x=1$, (3) all lines contained in a fixed plane, with 
parameter $x=1$ and (3) the disjoint union of (2) and (3), with parameter $x=2$. These examples 
and their complements are also known as {\em trivial examples} and are the only possibilities for their 
corresponding parameters. In the affine case, only examples (1) and (2) and the corresponding complements remain. 

Non-trivial examples of Cameron-Liebler line classes are rare. The first example of 
an infinite family was given in \cite{BruenAndDrude}. More recently, some infinite 
families have been described in \cite{DeBeule2016, Feng2015, Feng20xx}. It is 
noteworthy that the last three examples are affine, i.e. they are constructed 
in $\pg(3,q)$, but it turns out that there exists always a plane not containing any 
line of the Cameron-Liebler line class. Hence, see \cite[Theorem 3.8]{Me1}, 
these examples are also examples of non-trivial Cameron-Liebler line classes in 
$\ag(3,q)$. As mentioned in the introduction, non-existence results are of great 
interest as well, and one of the most consequential non-existence conditions
is Theorem~\ref{MetschRes} cited in Section~\ref{sec:intro}.

We now state the definition of Cameron-Liebler $k$-sets in general dimension. A point-pencil of $k$-spaces 
is the set of all $k$-spaces through a given point.

\begin{Def}
Let $n \geq 4$, and let $1 \leq k \leq n-1$. A Cameron-Liebler $k$-set in $\pg(n,q)$, is a set $\mathcal{L}$
of $k$-spaces such that its characteristic vector $\chi$ can be written as a linear combination of the characteristic vectors
of point-pencils. Its parameter $x$ is defined as $|\mathcal{L}|/\gauss{n}{k}_q$.
\end{Def}
\begin{Opm}
These objects are also know as boolean degree 1 functions, see  \cite{BDF, Ihringer2021}.
\end{Opm}

It can be shown that when there exists $k$-sreads in $\pg(n,q)$, a Cameron-Liebler $k$-set is characterized 
by its constant intersection property with $k$-spreads. The intersection number is then exactly its parameter, which 
is henceforth a natural number. Also note that the trivial examples of Cameron-Liebler line classes can be generalized
easily to Cameron-Liebler $k$-sets. 

When no $k$-spreads exist, the parameter of a Cameron-Liebler $k$-set is a rational number. The following theorem
provides more information on the parameter in this case, and it will turn out to be very useful.

\begin{St}\cite[Theorem 5.1 for $k=1$ and $t=3$]{Me3}\label{non-integer}
Suppose that $\mathcal{L}$ is a non-empty Cameron-Liebler line class in $\pg(n,q)$, $n \geq 4$ even, 
with parameter $x$. Then
\[
x=1+\frac{C}{\gauss{n-2}{1}_q}\,,
\]
for some $C \in \mathbb{N}$.
\end{St}

The following three lemmas will be of use in Section~\ref{sec:projective}. The set of $k$-spaces in the subspace $\pi$ (of dimension 
at least $k$), will be denoted as $[\pi]_k$. 

\begin{Le}(Folklore, \cite[Theorem 3.1]{Me3}) \label{CLSubspacePG}
Suppose that $\mathcal{L}$ is a Cameron-Liebler $k$-set in $\pg(n,q)$, $n>k\geq 1$. Then for every $i$-dimensional subspace $\pi$, with $i>k$ 
the set $\mathcal{L}\cap [\pi]_k$ is a Cameron-Liebler $k$-set of a certain parameter $x_\pi$ in $\pi$.
\end{Le}

\begin{Le}\cite[Theorem 2.9]{Jozefien} \label{SkewLines}
Suppose that $\mathcal{L}$ is a Cameron- Liebler line class with parameter $x$ in $\pg(n,q)$, with $n\geq 3$. If $\ell$ is an arbitrary line in $\pg(n,q)$ then there are in total $q^2\frac{q^{n-2}-1}{q-1}(x-\chi(\ell))$ lines of $\mathcal{L}$ skew to $\ell$.  Here $\chi(\ell)$ equals one if $\ell\in \mathcal{L}$ or zero otherwise.
\end{Le}

\begin{Le}\cite[Section 170]{Segre}\label{NumberOfDisjSpaces}
The number of $j$-spaces disjoint to a fixed $m$-space in $\pg(n,q)$ is equal to $q^{(m+1)(j+1)}\begin{bmatrix}
n-m \\
j+1
\end{bmatrix}_q$.
\end{Le}

In the affine case, we can easily define Cameron-Liebler $k$-sets using spreads. 

\begin{Def}
A Cameron-Liebler $k$-set with parameter $x$ in $\ag(n,q)$, $n \geq 4$ is a set $\mathcal{L}$ 
of lines of $\ag(n,q)$ such that $|\mathcal{L} \cap \mathcal{S}| = x$ for every $k$-spread $\mathcal{S}$
of $\ag(n,q)$.
\end{Def}

\begin{Ex}\label{ExSpread} 
A straightforward example of a $k$-spread in $\ag(n,q)$  consists of all skew $k$-spaces in $\ag(n,q)$ that have a $(k-1)$-space in the projective closure in common. In the case of lines this is often called a parallel class of lines.
\end{Ex}

\begin{Le}(Folklore, \cite[Theorem 3.6]{Me3}) \label{CLSubspace}
Suppose that $\mathcal{L}$ is a Cameron-Liebler $k$-set in $\ag(n,q)$, $n>k\geq 1$. Then for every $i$-dimensional subspace $\pi$, with $i>k$
the set $\mathcal{L}\cap [\pi]_k$ is a Cameron-Liebler $k$-set of a certain parameter $x_\pi$ in $\pi$.
\end{Le}

Finally, the following facts will be used to obtain the modular equations. 

\begin{Le}\label{Facts}
Let $a,x \in \mathbb{N}$. Let $q$ be a prime power. Then
\begin{itemize}
\item $2\frac{q^a-1}{q-1}\equiv 0 \mod 2(q+1)$ if $a\equiv 0 \mod 2$.
\item $aq^2\equiv a \mod 2(q+1)$ if $a\equiv 0 \mod 2$.
\item $x(x-1)\equiv 0 \mod 2$.
\end{itemize}
\end{Le}

\section{The affine case}

In this short section, we will generalize Theorem~\ref{GevKlausMain} to $\ag(n,q)$.

\begin{St}\label{Main1}
Suppose that $\mathcal{L}$ is a Cameron-Liebler line class in $\ag(n,q)$ with parameter $x$, then
$$x(x-1)\frac{q^{n-2}-1}{q-1}\equiv0 \mod 2(q+1).$$
\end{St}
\begin{proof}
Denote by  $\pi_{\infty}$ the hyperplane at infinity of the projective closure of $\ag(n,q)$. 
Fix a point $p \in \pi_{\infty}$. Now count triples 
$(\ell_1,\ell_2,\pi)$, for which 
\begin{itemize}
\item $\ell_1\cap \ell_2=p$
\item $\ell_1,\ell_2 \in \mathcal{L}$
\item $\langle \ell_1,\ell_2 \rangle \subset \pi$ and $\dim \pi=3$.
\end{itemize}
First, for any affine $3$-space $\pi$, the set of lines $\mathcal{L}\cap [\pi]_1$ 
is a Cameron-Liebler line class with parameter $x_{\pi}$ by Lemma~\ref{CLSubspace}. 
Since the lines of the parallel class in $\pi$ through $p$ are a line spread of 
$\pi$ (Example~\ref{ExSpread}), exactly $x_{\pi}$ of these lines are contained in 
$\mathcal{L}\cap [\pi]_1$. Hence there are $x_{\pi}(x_{\pi} - 1)$ choices for the pair $(\ell_1,\ell_2)$.

Secondy, for a fixed pair 
$(\ell_1,\ell_2)$, the number of affine $3$-spaces containing $\langle \ell_1,\ell_2 \rangle$ 
equals the number of $3$-spaces through a plane in $\pg(n,q)$, which is 
$\frac{q^{n-2}-1}{q-1}$.

Since the lines of the parallel class in $\pi$ through $p$ are a line spread of $\ag(n,q)$, 
there are $x(x-1)$ choices for a pair $(\ell_1,\ell_2)$ of lines of $\mathcal{L}$ through $p$. 
Hence, if we denote $\Phi_3$ as  the set of all affine $3$-spaces, we obtain that
\begin{equation*}
\sum_{\pi\in \Phi_3}x_\pi(x_\pi-1)= \sum_{(\ell_1,\ell_2)}\frac{q^{n-2}-1}{q-1} =  \frac{q^{n-2}-1}{q-1}x(x-1)\,.
\end{equation*}
Using Corollary \ref{GevKlausMain} on the affine space $\pi$, the left hand side reduces to $0 \mod{2(q+1)}$ and 
we indeed obtain the assertion.
\end{proof}
\begin{Opm}
The same result can be obtained by double counting the pairs $((\ell_1,\ell_1),\pi)$, 
with $p\in \ell_1, \ell_1 \cap \ell_2=\emptyset, \ell_1,\ell_2\in \cL$ and 
$\langle \ell_1, \ell_2\rangle=\pi$ a $3$-dimensional space. This proof is left to the reader.
\end{Opm}

Theorem~\ref{Main1} can be reformulated for Cameron-Liebler $k$-sets in $\ag(n,q)$, based on the 
following result.

\begin{St}\label{ToLines}\cite[Theorem 6.15]{Me2}
Let $\mathcal{L}$ be a Cameron-Liebler $k$-set in $\ag(n,q)$, with $n\geq k+2$. Suppose now that $\mathcal{L}$ 
has parameter $x$, then $x$ satisfies every condition which holds for Cameron-Liebler line classes in $\ag(n-k+1,q)$.
\end{St}

\begin{St}\label{Main2}
Let $\mathcal{L}$ be a Cameron-Liebler $k$-set in $\ag(n,q)$, with $n\geq k+5, k>1$, with parameter $x$. 
Then 
$$x(x-1)\frac{q^{n-k-1}-1}{q-1}\equiv 0 \mod 2(q+1).$$
\end{St}

Obviously, Theorem~\ref{Main2} becomes trivial for $n-k=3$.  Typical non-existence results give a lower or upper bound
on the parameter $x$. Theorem~\ref{Main2} can be combined with the following theorem.

\begin{St}\cite[Theorem 1.2]{Me3}\label{MainM3}
Suppose that $n\geq 2k+2$ and $k\geq 1$. Let $\mathcal{L}$ be a Cameron-Liebler $k$-set with parameter $x$ in $\ag(n,q)$ such that $\mathcal{L}$ is not a point-pencil, nor the empty set. Then
$$x\geq2\left(\frac{q^{n-k}-1}{q^{k+1}-1}\right)+1.$$
\end{St}
To show the significance of this theorem, we give the following small example.
\begin{Ex}
Let $n=5,k=1$ and $q=7$. We can restrict to $x\leq 1201\sim \frac12 q^{n-1}$ since the complement 
of a Cameron-Liebler line class is also a Cameron-Liebler line class. Theorem~ \ref{MainM3} implies 
$x\in \{0,1\}$ of $x \geq 101$. Theorem~\ref{Main2} gives $9x(x-1)\equiv 0 \mod 16$, which reduces the number of
possibilities for $x$ in a significant way. Finally note that for $q<7$ the classification is complete in \cite{BDF}.
\end{Ex}

\section{The projective case}\label{sec:projective}

In this section, we will prove a generalization of Theorem~\ref{MetschRes}. For a given
Cameron-Liebler line class $\mathcal{L}$ with parameter $x$ in $\pg(n,q)$ and a fixed point $p$,
the number of lines of $\mathcal{L}$ through $p$ will be denoted by $\overline{m}$.

\begin{St}\label{MainEq2}
Suppose that $\mathcal{L}$ is a Cameron-Liebler line class with parameter $x$ in $\pg(n,q)$, $n\geq 7$ odd, then
\[
x(x-1)+2\overline{m}(\overline{m}-x)\equiv 0 \mod (q+1)\,.
\]
\end{St}

Consider a Cameron-Liebler line class $\mathcal{L}$, a fixed point $p$, and a $3$-dimensional 
subspace $\pi$ through $p$. By Lemma~\ref{CLSubspacePG}, 
$\mathcal{L}_\pi:=\mathcal{L}\cap [\pi]_1$ is a Cameron-Liebler line class in $\pi$. Its parameter will be
denoted by $x_{\pi}$, and by $m_{\pi}$ we denote the number of lines $\mathcal{L}_\pi$ through $p$.

\begin{Le}\label{le:extra}
Suppose that $\mathcal{L}$ is a Cameron-Liebler line class in $\pg(n,q)$, $n \geq 4$. Then 
\[
\sum_{\pi\ni p} m_\pi=\frac{q^{n-1}-1}{q-1} \overline{m}\,.
\]
\end{Le}
\begin{Opm}
The sum of the equation above runs over all $3$-spaces $\pi$ 
through the point $p$. We will make the convention that the 
summation is always over the first object from the left.
\end{Opm}

\begin{proof}[Proof of Lemma \ref{le:extra}]
This follows by counting pairs $(\ell,\pi)$, where $\ell \in \mathcal{L}$, $p \in \ell$ and 
$\pi$ a $3$-dimensional space through $p$.
\end{proof}

\begin{Le}\label{Count1}
Suppose that $\mathcal{L}$ is a Cameron-Liebler line class in $\pg(n,q)$, $n \geq 4$. Let $p$
be a fixed point in $\pg(n,q)$. Then
$$\sum_{\pi\ni p}m_\pi^2= \overline{m}(\overline{m}-1)\frac{q^{n-2}-1}{q-1}+\frac{q^{n-1}-1}{q-1} \overline{m}$$
and
$$\sum_{\pi\ni p}m_\pi x_\pi  = \overline{m}(x-1)\frac{q^{n-2}-1}{q-1}+\frac{q^{n-1}-1}{q-1} \overline{m}$$
\end{Le}
\begin{proof}
First we count the triples $(\ell_1, \ell_2, \pi)$, $\ell_1, \ell_2 \in \mathcal{L}$, such that $\ell_1\cap \ell_2=p$, 
and $\ell_1, \ell_2 \subseteq \pi$ for a $3$-space $\pi$, both elements of $\mathcal{L}$. 
For a fixed $3$-dimensional subspace 
$\pi$, the number of possible pairs $(\ell_1, \ell_2)$ satisfying the conditions, equals $m_\pi(m_\pi-1)$. For a given
pair of lines $(\ell_1, \ell_2)$, there are exactly $\frac{q^{n-2}-1}{q-1}$ possible choices for $\pi$. The total
number of pairs $(\ell_1, \ell_2)$ equals $\overline{m}(\overline{m}-1)$. Hence 
\begin{equation*}
 \sum_{\pi\ni p}m_\pi(m_\pi-1) = \sum_{(\ell_1, \ell_2)}\frac{q^{n-2}-1}{q-1}  
= \overline{m}(\overline{m}-1)\frac{q^{n-2}-1}{q-1}\,.
\end{equation*}

Using Lemma~\ref{le:extra}, we find the first equation of the lemma,

\begin{equation}
\sum_{\pi\ni p}m_\pi^2 = \overline{m}(\overline{m}-1)\frac{q^{n-2}-1}{q-1}+\frac{q^{n-1}-1}{q-1} \overline{m}
\end{equation}

Secondly, we count the triples $(\ell_1, \ell_2, \pi)$, $\ell_1, \ell_2 \in \mathcal{L}$, such that 
$\ell_1\cap \ell_2=\emptyset$, $p\in \ell_1$ and $\ell_1, \ell_2 \subseteq \pi$, 
both elements of $\mathcal{L}$. For a fixed $3$-dimensional subspace $\pi$, 
we have $m_\pi$ possibilities for $\ell_1$ and, due to Lemma \ref{SkewLines}, we have $q^2(x_\pi-1)$ possibilities for $\ell_2$. 
For a fixed pair $(\ell_1, \ell_2)$ we have only one possible $3$-space $\pi$. 
Hence, 
\begin{equation*}
\sum_{\pi\ni p}m_\pi(x_\pi-1)q^2 = \sum_{(\ell_1, \ell_2)}1
\end{equation*}
We have in total $\overline{m}$ possibilities for $\ell_1$ and, 
by Lemma~\ref{SkewLines}, the number of lines of $\mathcal{L}$ skew to $\ell_1$ 
equals $(x-1)q^2\frac{q^{n-2}-1}{q-1}$. Now using Lemma~\ref{le:extra}, we can conclude that
\begin{equation}
\sum_{\pi\ni p}m_\pi x_\pi  = \overline{m}(x-1)\frac{q^{n-2}-1}{q-1}+\frac{q^{n-1}-1}{q-1} \overline{m}\\ 
\end{equation}
\end{proof}

For a Cameron-Liebler line class $\mathcal{L}$ in $\pg(n,q)$, Lemma~\ref{SkewLines} 
gives precise information on the number of lines of $\mathcal{L}$ skew to a given line 
of $\mathcal{L}$. In the of proof of Lemma~\ref{Count2}, we will need to control the number
of lines of $\mathcal{L}$ skew to a plane $\langle \ell_1,p \rangle$, $p \not \in \ell_1\in \cL$. It seems
hard to get precise information on this number. However, the next lemma provides us with
a modular equation on this number, which will be essential for Lemma~\ref{Count2}.

\begin{Le}\label{ModularT}
Suppose that $\mathcal{L}$ is a Cameron-Liebler line class in $\pg(n,q)$, $n \geq 7$ odd. Let $\pi$
be a fixed plane in $\pg(n,q)$ and let $\mathcal{L}\cap [\pi]_1\not=\emptyset$. Denote by
$T$ the number of lines of $\mathcal{L}$ skew to $\pi$. Then
$T\equiv 0 \mod (q+1)$.
\end{Le}
\begin{proof}
We will count the pairs $(\ell_2, \tau)$, $\tau$ a subspace of dimension $n-3$, 
$\ell_2 \in [\tau]_1 \cap \mathcal{L}$, and $\tau \cap \pi = \emptyset$. Consider a fixed line $\ell_2$. 
The number of $(n-3)$-spaces through $\ell_2$ skew to $\pi$ 
equals the number of $(n-5)$-spaces in $\pg(n-2,q)$ skew to a plane. By
Lemma~\ref{NumberOfDisjSpaces} this equals $q^{3(n-4)}$.
Since the lines of $\mathcal{L} \cap [\tau]_1$ induce a Cameron-Liebler
line class with parameter $x_{\tau}$ in $\tau$, and using Theorem~\ref{non-integer}, we find 
\begin{equation}\label{BeginEq}
q^{3(n-4)} T=\sum_{\tau} x_{\tau}\frac{q^{n-3}-1}{q-1} = \sum_{\tau} \left(1+\frac{C_{\tau}}{\frac{q^{n-2}-1}{q-1}}\right)\frac{q^{n-3}-1}{q-1}\,,
\end{equation}
where $C_{\tau} \in \mathbb{N}$ for each $\tau$. Hence 
\begin{equation}\label{NaturalEq}
\sum_{\tau}\frac{C_{\tau}\frac{q^{n-3}-1}{q-1}}{\frac{q^{n-2}-1}{q-1}} = q^{3(n-4)} T - \sum_{\tau} \frac{q^{n-3}-1}{q-1} \in \mathbb{N}\,,
\end{equation}
and it follows hat for some integer $K \in \mathbb{N}$, 
\[
\sum_{\tau}C_{\tau}\frac{q^{n-3}-1}{q-1}=\frac{q^{n-2}-1}{q-1}K\,.
\]
Since $\mathrm{gcd}(q^{n-2}-1,q^{n-3}-1) = q-1$, clearly $\frac{q^{n-2}-1}{q-1}\mid \sum_{\tau}C_{\tau}$, and now by 
Lemma \ref{Facts} (1),
$$\sum_{\tau}\frac{C_{\tau}\frac{q^{n-3}-1}{q-1}}{\frac{q^{n-2}-1}{q-1}} = 
\underbrace{\sum_{\tau }\frac{C_{\tau}}{\frac{q^{n-2}-1}{q-1}}}_{\in \mathbb{N}}\cdot \frac{q^{n-3}-1}{q-1} \equiv 0 \mod (q+1).$$

Thus, by Equation (\ref{BeginEq}), we indeed have that $T\equiv 0 \mod (q+1)$.
\end{proof}

\begin{Le}\label{Count2}
Suppose that $\mathcal{L}$ is a Cameron-Liebler line class in $\pg(n,q)$, $n \geq 7$ odd. Let $p$
be a fixed point in $\pg(n,q)$. Then
\[
2\sum_{\pi \ni p} (x_\pi-1)x_\pi\equiv 2\frac{q^n-1}{q-1}\frac{q^{n-2}-1}{q-1}x(x-1)\mod 2(q+1)\,.
\]
\end{Le}
\begin{proof}
We will count the triples $(\ell_1, \ell_2, \pi)$,  $\ell_1, \ell_2 \in \mathcal{L}$, such that for the lines $\ell_1, \ell_2$, 
$\ell_1\cap \ell_2=\emptyset$, $p\not\in \ell_1$, $\ell_2\cap  \langle \ell_1,p\rangle \neq \emptyset$ 
and $\pi = \langle \ell_1, \ell_2 \rangle$. 

Fix a $3$-dimensional subspace $\pi$. Then there are 
$|\mathcal{L}_\pi|-m_\pi$ choices for $\ell_1$. For a chosen $\ell_1$, by Lemma~\ref{SkewLines}, there are 
$(x_\pi-1)q^2$ suitable lines skew to $\ell_1$. For a fixed pair of lines $(\ell_1, \ell_2)$, the $3$-dimensional space
$\pi$ is uniquely determined. Hence, with $N$ the number of pairs $(\ell_1,\ell_2)$ satisfying the conditions, 
\begin{equation}\label{Counting3}
N = \sum_{\pi\ni p} (x_\pi-1)q^2(x_\pi(q^2+q+1)-m_\pi)\,.
\end{equation}
By Lemma~\ref{SkewLines}, there are $(x-1)q^2\frac{q^{n-2}-1}{q-1}$ lines of $\mathcal{L}$ skew to 
a given line $\ell_1$. Denote by $T_{\ell_1}$ the number of lines of $\mathcal{L}$ skew to 
$\langle \ell_1,p \rangle$. Then
\begin{equation}\label{Counting3bis}
N = \sum_{\ell_1 \in \mathcal{L}, p \not \in \ell_1} \left((x-1)q^2\frac{q^{n-2}-1}{q-1}-T_{\ell_1} \right)
\end{equation}

Now we can reduce the expressions \eqref{Counting3} and \eqref{Counting3bis} modulo $2(q+1)$
and get

\[
\sum_{\pi\ni p} (x_\pi-1)x_\pi q^2(q^2+q+1)-\sum_{\pi\ni p} q^2 (x_\pi-1)m_\pi 
= \sum_{\ell_1} (x-1)q^2\frac{q^{n-2}-1}{q-1}-\sum_{\ell_1}T_{\ell_1} \mod 2(q+1)\,.
\]

Note that $q^2+q+1 = q(q+1) + 1$. Since $(x_\pi-1)x_\pi$ is even, $2(q+1) \mid (x_\pi-1)x_\pi q(q+1)$, and by
Lemma~\ref{Facts} it follows that $(x_\pi-1)x_\pi q^2 \equiv (x_\pi-1)x_\pi \mod 2(q+1)$. Hence we find

\begin{equation}\label{eq:tussen}
\sum_{\pi\ni p} (x_\pi-1)x_\pi-\sum_{\pi\ni p} q^2 (x_\pi-1)m_\pi  \equiv 
\sum_{\ell_1} (x-1)q^2\frac{q^{n-2}-1}{q-1}-\sum_{\ell_1}T_{\ell_1} \mod 2(q+1)\\
\end{equation}

Note that there are $|\mathcal{L}| - \overline{m} =x\frac{q^n-1}{q-1} - \overline{m}$ candidates for $\ell_1$. Furthermore,
by Lemma \ref{Count1},
\[
\sum_{\pi\ni p} q^2 (x_\pi-1)m_\pi=\overline{m}(x-1)q^2\frac{q^{n-2}-1}{q-1}\,.
\]

Hence Equation~\ref{eq:tussen} reduces to 

\begin{equation}\label{eq:tussen2}
\sum_{\pi\ni p} (x_\pi-1)x_\pi \equiv q^2\frac{q^n-1}{q-1}\frac{q^{n-2}-1}{q-1}x(x-1)-\sum_{\ell_1}T_{\ell_1} \mod 2(q+1)
\end{equation}

From Lemma~\ref{ModularT} it follows that $2 T_{\ell_1} \equiv 0 \mod 2(q+1)$. 
So multiplying Equation (\ref{eq:tussen2}) with $2$, and combined with Lemma~\ref{Facts} we obtain
\[
2\sum_{\pi\ni p} (x_\pi-1)x_\pi\equiv 2\frac{q^n-1}{q-1}\frac{q^{n-2}-1}{q-1}x(x-1)\mod 2(q+1)\,.
\]
\end{proof}

Now we are ready to prove the main theorem of this section.

\begin{St}\label{MainEq2-final}
Suppose that $\mathcal{L}$ is a Cameron-Liebler line class with parameter $x$ in $\pg(n,q)$, with $n\geq 7$ odd, then
\[
x(x-1)+2\overline{m}(\overline{m}-x)\equiv 0 \mod (q+1)\,.
\]
\end{St}

\begin{proof}
Assume that $\pi$ is a $3$-space. Then the line set $\cL_\pi=\mathcal{L}\cap[\pi]_1$ is a Cameron-Liebler 
line class in $\pi$ of a certain parameter $x_\pi$. By Theorem~\ref{MetschRes}, 
\[
x_\pi(x_\pi-1)+2m_\pi(m_\pi-x_\pi)\equiv 0 \mod 2(q+1)\,.
\]

So in particular, for a fixed point $p$,

\begin{equation*}
\begin{split}
\sum_{\pi\ni p} \left( x_\pi(x_\pi-1)+2m_\pi(m_\pi-x_\pi)\right)&\equiv0 \mod 2(q+1) \\
\Leftrightarrow \sum_{\pi\ni p}  x_\pi(x_\pi-1)+2\sum_{\pi\ni p}m_\pi^2-2\sum_{\pi\ni p}m_\pi x_\pi&\equiv0 \mod 2(q+1) \\
\end{split}
\end{equation*}
Filling in the equations from Lemma \ref{Count1}, we have that
$$\sum_{\pi\ni p}  x_\pi(x_\pi-1)+ 2\frac{q^{n-2}-1}{q-1}\overline{m}(\overline{m}-x)\equiv0 \mod 2(q+1) .$$
Multiplying this equation by $2$ and using Lemma \ref{Count2}, we find
\begin{equation*}
\begin{split}
2\frac{q^n-1}{q-1}\frac{q^{n-2}-1}{q-1}x(x-1)+4\frac{q^{n-2}-1}{q-1}\overline{m}(\overline{m}-x)&\equiv0 \mod 2(q+1)\\
\Rightarrow\frac{q^n-1}{q-1}\frac{q^{n-2}-1}{q-1}x(x-1)+2\frac{q^{n-2}-1}{q-1}\overline{m}(\overline{m}-x)&\equiv0 \mod (q+1)\\
\end{split}
\end{equation*}
Now for $n \geq 7$ odd, we have 
$$\frac{q^{n-2}-1}{q-1}\equiv \frac{q^n-1}{q-1}\equiv 1 \mod (q+1).$$
Hence we obtain
$$x(x-1)+2\overline{m}(\overline{m}-x)\equiv 0 \mod (q+1),$$
which proves the statement of the theorem.
\end{proof}

The following corollary is easy to prove using the principle of duality in $\pg(n,q)$.

\begin{Gev}
Suppose that $\mathcal{L}$ is a Cameron-Liebler $(n-2)$-set with parameter $x$ in $\pg(n,q)$, $n\geq 7$ odd, then
$$x(x-1)+2\overline{m}(\overline{m}-x)\equiv 0 \mod (q+1).$$
Here $\overline{m}$ denotes the number of $(n-2)$-spaces of $\mathcal{L}$ inside a fixed hyperplane $\pi$.
\end{Gev}


\section{Final remarks}

In this final section we compare both modular conditions in $\ag(n,q)$. Since a Cameron-Liebler line class
in $\ag(n,q)$, induces also a Cameron-Liebler line class in $\pg(n,q)$ with the same parameter (see \cite[Theorem 2]{Me2}),
its parameter must satisfy Theorems~\ref{Main1} and
\ref{MainEq2}.
 
Let $\mathcal{L}$ be a Cameron-Liebler line class with parameter $x$ in $\ag(n,q)$, then by Theorem~\ref{Main1}
\begin{equation}\label{C1}
\frac{q^{n-2}-1}{q-1}x(x-1)\equiv 0 \mod 2(q+1).
\end{equation}
Now let $n \geq 7$, odd, then by Theorem~\ref{MainEq2}
\begin{equation}\label{C2}
x(x-1)+2\overline{m}(\overline{m}-x)\equiv 0 \mod (q+1),
\end{equation}
where $\overline{m}$ is the number of lines of $\mathcal{L}$ through a chosen point $p$. Choosing 
$p$ at infinity, gives $\overline{m} = x$, and hence only weaker information compared with the first equation is obtained.
To compute $\overline{m}$ for affine points, we need to do some more work, using 
a similar strategy as in \cite[Lemma 4.4]{Me1}. First we need the following lemma.

\begin{Le}\cite[Lemma 2.12]{Jozefien}\label{DrudgeArgum}
Let $\mathcal{L}$ be a Cameron-Liebler $k$-set in $\pg(n,q)$, then 
for every point $p$ and every $i$-dimensional subspace $\tau$, 
with $p\in \tau$ and $i\geq k+1$, 
\begin{equation}\label{EqIntersections}
\left| [p]_k\cap \mathcal{L} \right|+ \frac{\gauss{n-1}{k}_q (q^k-1)}{\gauss{i-1}{k}_q (q^i-1)} \left| [\tau]_k\cap \mathcal{L} \right| 
= \frac{\gauss{n-1}{k}_q }{\gauss{i-1}{k}_q } \left| [p,\tau]_k\cap \mathcal{L} \right| + \frac{q^k-1}{q^n-1} \left| \mathcal{L} \right|\,.
\end{equation}
\end{Le}

Using this lemma, we find the following result in Lemma \ref{le:compare}. It shows that Equation~\eqref{C2} 
considered in $\ag(n,q)$ { will always reduce to $x(x-1)\equiv 0 \mod(q+1)$, which is} a weaker condition than Equation~\ref{C1} 
for {\bf affine} Cameron-Liebler line classes. 
\begin{Le}\label{le:compare}
Suppose that $\mathcal{L}$ is a Cameron-Liebler line class in $\ag(n,q)$, for $n\geq 3$ odd, then 
for any hyperplane $\pi$ and any point $p$,
\[
|[\pi]_1\cap\mathcal{L}|\equiv 0 \mod (q+1) \text{ and } |[p]_1\cap\mathcal{L}|\equiv x \mod (q+1)\,.
\]
\end{Le}
\begin{proof}
Choose an hyperplane $\pi$ and fix a point $p \in \pi$ at infinity in the closure of $\ag(n,q)$. 
By Lemma~\ref{DrudgeArgum}, $\left| [p]_k\cap \mathcal{L} \right|=x$. Using this in Equation~\eqref{EqIntersections}
and multiplying with $\frac{q^{n-2}-1}{q-1}$, we obtain, for $i=n-1$ and $k=1$
\[
\left| [\pi]_1\cap \mathcal{L} \right|  = \frac{q^{n-1}-1}{ q-1}\left| [p,\pi]_1\cap \mathcal{L}\right| \,.
\]
Now by Lemma~\ref{Facts}, if follows for $n$ odd that 
\[ \left| [\pi]_1\cap \mathcal{L} \right| = \frac{q^{n-1}-1}{ q-1}\left| [p,\pi]_1\cap \mathcal{L} \right|\equiv 0 \mod (q+1)\,.
\]
This proves the first part of the lemma.
Secondly, choose an arbitrary affine point $p$, and pick an hyperplane $\pi$ through $p$. Using our previous observations, we obtain that $|[\pi]_1\cap\mathcal{L}|\equiv 0 \mod (q+1) $. If we take a look at Equation (\ref{EqIntersections}), and first multiply this equation with $\frac{q^{n-2}-1}{q-1}$, we obtain that:
$$\left|[p]_1\cap \mathcal{L}\right|\frac{q^{n-2}-1}{q-1}=\frac{q^{n-1}-1}{q-1}\left|[p,\pi]_1\cap \mathcal{L}\right|+ \frac{q^{n-2}-1}{q-1}x.$$
Using Lemma \ref{Facts} and noticing that $n$ is odd, we obtain indeed that
$$\left|[p]_1\cap\mathcal{L}\right|\equiv x \mod (q+1).$$
This proves the assertion.
\end{proof}

\bibliographystyle{plain}

\end{document}